\documentclass[11pt, a4paper]{article}

\usepackage{amsmath, amssymb, amsthm}
\usepackage{geometry}
\usepackage[T1]{fontenc}
\usepackage{lmodern}
\usepackage{indentfirst}
\usepackage{xcolor}
\usepackage[colorlinks=true, linkcolor=blue, citecolor=red, urlcolor=blue]{hyperref}

\usepackage{tikz}
\usetikzlibrary{shapes.geometric, arrows.meta, positioning, fit, backgrounds}

\providecommand{\citeauthoryear}[3]{}

\newtheorem{theorem}{Theorem}[section]
\newtheorem{lemma}[theorem]{Lemma}
\newtheorem{corollary}[theorem]{Corollary}
\newtheorem{definition}[theorem]{Definition}
\newtheorem{remark}[theorem]{Remark}

\begin{document}

\title{\textbf{A Container Theorem for General Digraphs \\ via Pair-State Encoding}}

\author{
    Meili Liang, Yue Guan, Ruiling Zheng, Jianxi Liu\thanks{Corresponding author. Email: jxliu@gdufs.edu.cn} \\
    \small School of Mathematics and Statistics, Guangdong University of Foreign Studies, Guangzhou, China.
}

\date{} 

\maketitle

\begin{abstract}
In a seminal work, K\"uhn, Osthus, Townsend, and Zhao used the hypergraph container method to determine the typical structure of oriented graphs avoiding a fixed tournament or cycle. Direct extensions of their theorem to general digraphs fail due to dense local obstructions like 2-cycles, which cause a co-degree explosion (a division-by-zero bottleneck) in the standard directed-edge hypergraph model. 

In this paper, we completely resolve this bottleneck by introducing a pair-state encoding framework. By shifting the ground set of the hypergraph from directed edges to unordered vertex pairs endowed with directional states, we structurally guarantee that any subgraph consisting of at least two pairs spans at least three vertices. This bypasses the division-by-zero obstacle entirely and yields a universal container theorem for all digraphs without requiring any artificial sparsity conditions. As applications, we obtain optimal asymptotic counting results for $H$-free digraphs for any general forbidden digraph $H$.

\end{abstract}

\section{Introduction}\label{sec1}

The problem of enumerating and describing the typical structure of graphs with forbidden subgraphs has a long history in extremal combinatorics. For directed graphs, substantial progress was made by K\"uhn, Osthus, Townsend and Zhao \cite{KOTZ} (hereafter KOTZ), who developed a container theorem for \emph{oriented graphs} (digraphs with no 2-cycles). 

As noted in \cite{KOTZ}, their container theorem cannot be directly extended to general digraphs using the standard edge-based uniform hypergraph method. The obstruction comes from the high local density of 2-cycles. In the standard application of the hypergraph container method to graphs or oriented graphs \cite{BaloghMorrisSamotij, SaxtonThomason}, the ground set consists of the $N(N-1)$ directed edges. If the forbidden digraph $H$ contains a 2-cycle, this 2-cycle constitutes a sub-configuration $H'$ with $e(H') = 2$ directed edges but only $v(H') = 2$ vertices. The crucial maximum density parameter $m(H)$ in the hypergraph container method requires bounding expressions of the form $\frac{e(H')-1}{v(H')-2}$. When $v(H') = 2$, this leads to a division-by-zero ($2-2=0$) and causes the co-degree estimates to explode, forcing the container penalty factor to become trivial.

To visually illustrate this bottleneck, Figure~\ref{fig:comparison} contrasts the graph classes handled by these frameworks. While the standard directed-edge model (e.g., KOTZ 2017) is strictly limited to oriented graphs where 2-cycles are prohibited (left), our approach universally accommodates general digraphs (right). The highlighted 2-cycle represents the dense local obstruction that triggers the fatal division-by-zero issue in the classical uniform hypergraph setup.

\begin{figure}[htbp]
    \centering
    \resizebox{0.85\textwidth}{!}{%
    \begin{tikzpicture}[>=stealth, node distance=2cm, thick]
        \node[circle, draw, minimum size=0.6cm] (A1) at (0, 0) {};
        \node[circle, draw, minimum size=0.6cm] (B1) at (2, 0) {};
        \node[circle, draw, minimum size=0.6cm] (C1) at (1, 1.732) {};
        
        \draw[->] (A1) -- (B1);
        \draw[->] (B1) -- (C1);
        \draw[->] (A1) -- (C1);
        
        \node[align=center] at (1, -1) {\textbf{Oriented Graphs}\\ (KOTZ 2017)};
        \node[align=center, text=gray, font=\small] at (1, -1.6) {No 2-cycles allowed};

        \draw[loosely dashed, gray] (2.8, 2.5) -- (2.8, -2);

        \node[circle, draw, minimum size=0.6cm] (A2) at (4, 0) {};
        \node[circle, draw, minimum size=0.6cm] (B2) at (6, 0) {};
        \node[circle, draw, minimum size=0.6cm] (C2) at (5, 1.732) {};
        
        \draw[->, bend left=20] (A2) to (B2);
        \draw[->, bend left=20] (B2) to (A2);
        
        \draw[->] (B2) -- (C2);
        \draw[->] (C2) to (A2);
        
        \draw[densely dashed, rounded corners] (3.6, -0.5) rectangle (6.4, 0.5);
        \node[font=\footnotesize] at (8.5, 0) {Local Obstruction: 2-cycle};
        
        \node[align=center] at (5, -1) {\textbf{General Digraphs}\\ (Our Work)};
        \node[align=center, text=gray, font=\small] at (5, -1.6) {Universal: Contains 2-cycles};
    \end{tikzpicture}%
    }
    \caption{Comparison of graph classes handled by the hypergraph container methods. The standard directed-edge model (KOTZ 2017) fails when 2-cycles are present due to a division-by-zero bottleneck.}
    \label{fig:comparison}
\end{figure}

To resolve the constraints imposed by multi-edge densities and state-dependent adjacencies, Falgas-Ravry, O'Connell, and Uzzell \cite{FalgasRavry} developed a foundational framework known as the \emph{multicolour container theory}. They generalized the classic monochromatic container theorems \cite{BaloghMorrisSamotij, SaxtonThomason} by allowing the underlying elements of the hypergraph to carry multiple categorical colors or states. By integrating this multicolour container lemma with an entropy-based framework on decorated graphons, they successfully characterized general hereditary properties and established robust counting results for diverse structures, including general digraphs, tournaments, and multigraphs with bounded multiplicities. 

Our approach shares a profound philosophical insight with the work of Falgas-Ravry et al.\ \cite{FalgasRavry} regarding the elimination of the local density obstruction. Both frameworks recognize that directed edges cannot serve as independent entities in the hypergraph ground set when dense 2-cycles are forbidden. Consequently, both approaches build upon the geometric foundation of \emph{unordered vertex pairs} coupled with relational states (or colors). Because any sub-configuration consisting of at least two distinct unordered pairs structurally spans at least 3 vertices, the quantity $v(H')-2 \ge 1$ remains strictly positive, effectively neutralizing the co-degree explosion bottleneck.

However, significant distinctions arise in the theoretical abstraction, the complexity of implementation, and the formulation of density parameters. The framework by Falgas-Ravry et al.\ \cite{FalgasRavry} adopts a top-down, highly abstract pathway heavily rooted in measure theory and continuous graphon limits, where the asymptotic bounds are implicitly expressed through variational optimization of extremal entropy. While theoretically elegant and broad, it often requires substantial machinery to evaluate for specific graph properties. In contrast, the present work introduces a bottom-up, purely combinatorial route. Instead of resetting the foundation to establish a new multicolour container lemma, we invent a tactical map that translates the multi-state relational constraints back into the classical monochromatic, uniform hypergraph model of Saxton and Thomason \cite{SaxtonThomason}. This is achieved seamlessly by expanding the ground set to $U = X \times \{F, B, D\}$ and enforcing a unique \emph{downward-closure property} on the encoded representation. 

To illustrate this bottleneck and our resolution concretely, consider the complete digraph on three vertices, denoted by $DK_3$, which contains all three possible 2-cycles (amounting to six directed edges). Under the classical directed-edge hypergraph framework of KOTZ \cite{KOTZ}, any attempt to build containers for $DK_3$-free digraphs immediately encounters a fatal barrier. A sub-configuration $H' \subseteq DK_3$ consisting of a single 2-cycle yields $v(H') = 2$ and $e(H') = 2$, forcing the standard density denominator $v(H') - 2$ to become zero. This division-by-zero causes the co-degree estimates to explode and renders the classical container theorem completely inapplicable to $DK_3$. In contrast, our pair-state encoding framework seamlessly accommodates $DK_3$. Here, the forbidden digraph has $p(DK_3) = 3$ underlying unordered pairs, which satisfies our basic requirement $p(H) \ge 2$. For any sub-configuration $H' \subseteq DK_3$ with $p(H') \ge 2storage$, any two distinct unordered pairs structurally must span at least $s \ge 3$ vertices. This geometric guarantee ensures that $v(H') - 2 \ge 1$ is strictly positive for all valid sub-configurations. Consequently, the pair-density parameter is perfectly well-defined and finite, evaluated as $m_p(DK_3) = \max \{ \frac{2-1}{3-2}, \frac{3-1}{3-2} \} = 2$. Applying Theorem~\ref{thm:main}, we immediately obtain a sharp container upper bound of $\log |\mathcal{C}| \le c N^{3/2} \log N$, demonstrating that our bottom-up framework successfully tames dense local obstructions without modifying the underlying monochromatic hypergraph machinery.

Furthermore, we define a highly explicit and intuitive graph theoretic parameter---the \emph{pair-density} $m_p(H)$ (Definition \ref{def:mp}). This parameter precisely captures the polynomial growth rate of co-degrees under the state encoding, allowing us to directly establish the container size bound $O(N^{2-1/m_p(H)}\log N)$ in a sharp, clear, and constructive manner. Thus, our method provides an algorithmic and transparent alternative for analyzing general forbidden subdigraphs.

\begin{definition}[Pair-size and Pair-density]\label{def:mp}
For a digraph $H$, let $p(H)$ denote the number of unordered pairs of vertices in $H$ that contain at least one directed edge. We define the pair-density parameter $m_p(H)$ as:
\[
m_p(H) = \max_{H' \subseteq H, p(H') \ge 2} \frac{p(H') - 1}{v(H') - 2}.
\]
\end{definition}

\begin{remark}
Since any graph with $p(H') \ge 2$ distinct edges on unordered pairs requires at least 3 vertices (the maximum number of pairs on 2 vertices is $\binom{2}{2}=1$), the denominator $v(H') - 2 \ge 1$ is always strictly positive. Thus, $m_p(H)$ is well-defined and finite for all digraphs with $p(H) \ge 2$.
\end{remark}

Our main result is the following universal container theorem for digraphs.

\begin{theorem}[Container Theorem for General Digraphs]\label{thm:main}
Let $H$ be a general digraph with $p(H) \ge 2$ and let $a \ge 1$. For every $\varepsilon > 0$, there exists $c > 0$ such that for all sufficiently large $N$, there exists a collection $\mathcal{C}$ of digraphs (containers) on vertex set $[N]$ with the following properties:
\begin{enumerate}
    \item[(a)] Every $H$-free digraph $I$ on $[N]$ is a subgraph of some $G \in \mathcal{C}$.
    \item[(b)] Every $G \in \mathcal{C}$ contains at most $\varepsilon N^{v(H)}$ copies of $H$, and $e_a(G) \le \operatorname{ex}_a(N, H) + \varepsilon N^2$.
    \item[(c)] $\log |\mathcal{C}| \le c N^{2-1/m_p(H)} \log N$.
\end{enumerate}
\end{theorem}

As a direct consequence of Theorem \ref{thm:main}, we obtain an asymptotic enumeration result for $H$-free digraphs for any general forbidden digraph $H$, filling a gap that previously required specialized sparsity conditions.

\begin{corollary}[Asymptotic Counting of $H$-Free Digraphs]\label{cor:counting}
Let $H$ be a general digraph with $p(H) \ge 2$. The total number of $H$-free digraphs on $N$ vertices, denoted by $f^*(N, H)$, satisfies
\[
\log_2 f^*(N, H) = \operatorname{ex}_2(N, H) + o(N^2).
\]
\end{corollary}

\section{Notation and Preliminaries}\label{sec2}

For a digraph $G = (V, E)$, we distinguish between single arcs and 2-cycles. Let $f_2(G)$ denote the number of 2-cycles (pairs of vertices between which both directed arcs are present) and $f_1(G)$ denote the number of single arcs (pairs between which exactly one directed arc is present). The generalized edge count of $G$ with respect to a parameter $a \ge 1$ is defined as:
\[
e_a(G) = a \cdot f_2(G) + f_1(G).
\]
Taking $a=2$ corresponds exactly to counting the total number of directed edges in the digraph. 

The extremal number $\operatorname{ex}_a(N, H)$ is defined as the maximum value of $e_a(G)$ among all $H$-free digraphs $G$ on $N$ vertices:
\[
\operatorname{ex}_a(N, H) = \max \{ e_a(G) : G \text{ is } H\text{-free}, |V(G)| = N \}.
\]

Our proof relies on the standard uniform hypergraph container theorem established by Saxton and Thomason \cite{SaxtonThomason} and independently by Balogh, Morris, and Samotij \cite{BaloghMorrisSamotij}. To state their result, we introduce the necessary notation for hypergraph co-degrees. For an $r$-uniform hypergraph $\mathcal{H}$ on $n$ vertices, and a subset $\sigma \subseteq V(\mathcal{H})$, let $d(\sigma)$ denote the number of edges containing $\sigma$. For $1 \le j \le r$, let $\Delta_j(\mathcal{H})$ be the maximum degree of a subset of size $j$. The co-degree function is defined as:
\[
\delta(\mathcal{H}, \tau) = 2^{\binom{r}{2}-1} \sum_{j=2}^r 2^{-(j-1)} \delta_j,
\]
where the parameters $\delta_j$ are given by the relation:
\[
\delta_j \tau^{j-1} n \Delta_1 = \sum_{v \in V(\mathcal{H})} \Delta_j(\{v\}),
\]
and $\Delta_j(\{v\})$ represents the maximum number of edges containing $\{v\} \cup \sigma'$ over all subsets $\sigma' \subseteq V(\mathcal{H}) \setminus \{v\} $ of size $j-1$.

\begin{theorem}[Saxton--Thomason \cite{SaxtonThomason}]\label{thm:ST}
Let $0 < \varepsilon < 1/2$ and let $\mathcal{H}$ be an $r$-uniform hypergraph on $n$ vertices. Let $\tau $ be a positive real number such that $\tau \le 1/(144r^2 r!)$. Suppose that the co-degree function satisfies:
\[
\delta(\mathcal{H}, \tau) \le \frac{\varepsilon}{12\tau}.
\]
Then there exists a collection $\mathcal{C}^*$ of subsets of $V(\mathcal{H})$ such that:
\begin{enumerate}
    \item Every independent set of $\mathcal{H}$ is contained in some $C^* \in \mathcal{C}^*$.
    \item For every $C^* \in \mathcal{C}^*$, the number of induced edges satisfies $e(\mathcal{H}[C^*]) \le \varepsilon e(\mathcal{H})$.
    \item The size of the collection satisfies $\log |\mathcal{C}^*| \le c(r) n \tau \log(1/\tau)$, where $c(r)$ is a constant depending only on $r$.
\end{enumerate}
\end{theorem}

\section{The Pair-State Encoding and the Auxiliary Hypergraph}

To apply Theorem~\ref{thm:ST} without triggering the 2-cycle division-by-zero bottleneck, we construct an auxiliary hypergraph $\mathcal{H}$ over a state-encoded ground set. Let $X = \binom{[N]}{2}$ be the set of all $\binom{N}{2}$ unordered vertex pairs from the vertex set $[N]$. We define our ground set $U$ as the Cartesian product of $X$ and a set of three directional states $\Sigma = \{F, B, D\}$, representing Forward, Backward, and Double, respectively:
\[
U = X \times \{F, B, D\}.
\]
Note that the number of vertices in our auxiliary hypergraph is $n = |U| = 3\binom{N}{2} = \Theta(N^2)$.

\subsection{Encoding Digraphs into the Ground Set}
For any digraph $G$ on $[N]$, we define its encoded representation $S_G \subseteq U$ via a downward-closure property. For each unordered pair $e = \{u, v\}$ with $u < v$, we populate $S_G$ according to the following rules:
\begin{itemize}
    \item If $G$ contains only the single arc $u \to v$, we add the element $(e, F)$ to $S_G$.
    \item If $G$ contains only the single arc $v \to u$, we add the element $(e, B)$ to $S_G$.
    \item If $G$ contains a 2-cycle on $e$ (meaning both $u \to v$ and $v \to u$ are present), we add \textbf{all three elements} $(e, D)$, $(e, F)$, and $(e, B)$ to $S_G$.
\end{itemize}
This downward-closure property ensures that a 2-cycle implicitly satisfies any subgraph requirement for a single directed edge in either direction, which is structurally essential for the hypergraph encoding to reflect subgraph containment correctly.

\subsection{Defining the Hyperedges}
We define an $r$-uniform hypergraph $\mathcal{H}$ with vertex set $V(\mathcal{H}) = U$, where the uniformity is set to $r = p(H)$ (the number of underlying pairs with edges in the forbidden digraph $H$). Let $h = v(H)$. A labeled copy of $H$ on $[N]$ is defined by an injection $\phi: V(H) \to [N]$ that preserves adjacency. For each such injection $\phi$, we construct a hyperedge $E_\phi \subset U$ of size $r$ by evaluating each of the $r$ interacting vertex pairs in $H$. For each pair $\{x, y\} \subset V(H)$ that contains at least one edge in $H$, let us assume without loss of generality that $\phi(x) < \phi(y)$ and denote $e = \{\phi(x), \phi(y)\}$. We add exactly one element to $E_\phi$ based on the structure of $H$ on $\{x,y\}$:
\begin{itemize}
    \item If $H$ has a 2-cycle on $\{x,y\}$, we add $(e, D)$ to $E_\phi$.
    \item If $H$ has only the single arc $x \to y$, we add $(e, F)$ to $E_\phi$.
    \item If $H$ has only the single arc $y \to x$, we add $(e, B)$ to $E_\phi$.
\end{itemize}
Since $E_\phi$ contains exactly one element for each of the $r$ pairs, it is an $r$-uniform hypergraph. The downward-closure definition of $S_G$ ensures that if $G$ contains the copy of $H$ given by $\phi$, then all elements of $E_\phi$ are present in $S_G$. Conversely, if $E_\phi \subseteq S_G$, the definition of $S_G$ implies that $G$ possesses the required arcs. This yields the following fundamental equivalence:

\begin{lemma}\label{lem:independence}
A digraph $G$ contains a copy of $H$ if and only if its encoded representation $S_G$ contains a hyperedge $E_\phi \in \mathcal{H}$. Consequently, $G$ is $H$-free if and only if $S_G$ is an independent set in $\mathcal{H}$.
\end{lemma}

\section{Degree Estimates without Division by Zero}

In this section, we establish the bounded co-degree property for the auxiliary hypergraph $\mathcal{H}$, proving that the pair-state encoding successfully bypasses the division-by-zero obstacle.

\begin{lemma}\label{lem:degree}
For any $\gamma > 0$ and sufficiently large $N$, the co-degree function satisfies $\delta(\mathcal{H}, \tau) \le O(1)$ with the choice of container parameter $\tau = \gamma N^{-1/m_p(H)}$.
\end{lemma}

\begin{proof}
The total number of hyperedges in $\mathcal{H}$ corresponds to the number of labeled embeddings of $H$ into $[N]$, which yields $e(\mathcal{H}) = \Theta(N^h)$. Since the number of vertices in $\mathcal{H}$ is $n = \Theta(N^2)$ and $\mathcal{H}$ is $r$-uniform, the average degree of a vertex in $U$ is given by $d = \frac{r e(\mathcal{H})}{n} = \Theta(N^{h-2})$. 

The maximum degree $\Delta_1$ is bounded by fixing one element in $U$, which fixes the coordinates of 2 vertices of $H$, and then choosing the remaining $h-2$ vertices from $[N]$. Thus, $\Delta_1 = \Theta(N^{h-2})$, which directly implies that $n \Delta_1 = \Theta(N^2) \cdot \Theta(N^{h-2}) = \Theta(N^h)$.

Now, let us consider a subset $\sigma \subset U$ of size $j \ge 2$. If $\sigma$ contains conflicting states for the same underlying unordered pair (for example, both $(e, F)$ and $(e, B)$, or $(e, F)$ and $(e, D)$), then by construction, no hyperedge $E_\phi$ can contain $\sigma$, because each hyperedge contains exactly one state per pair. In this case, the co-degree satisfies $d(\sigma) = 0$.

Therefore, we only need to consider subsets $\sigma$ that do not contain conflicting states. Such a subset $\sigma$ consists of $j$ distinct unordered pairs in $[N]$. Let $s$ denote the number of vertices in $[N]$ spanned by these $j$ pairs. Since the maximum number of pairs that can be formed on $s$ vertices is $\binom{s}{2}$, we must have the structural inequality:
\[
j \le \binom{s}{2}.
\]
For $j \ge 2$, this inequality strictly implies that $s \ge 3$. Consequently, we have $s - 2 \ge 1 > 0$, which structurally eliminates the division-by-zero bottleneck. 

By Definition~\ref{def:mp}, the sub-configuration $H'$ of $H$ induced by these $j$ pairs satisfies the pair-density bound $\frac{p(H') - 1}{v(H') - 2} \le m_p(H)$. Since $p(H') = j$ and $v(H') = s$, this can be rewritten as $\frac{j - 1}{s - 2} \le m_p(H)$, which rearranges to:
\[
s - 2 \ge \frac{j-1}{m_p(H)}.
\]
The number of ways to extend this sub-configuration $\sigma$ to a full labeled copy of $H$ is bounded by the number of choices for the remaining $h - s$ vertices from $[N]$. Therefore, the co-degree of $\sigma$ is bounded by:
\[
d(\sigma) \le O(N^{h-s}) \le O\left(N^{h - 2 - (j-1)/m_p(H)}\right).
\]
Crucially, because the inequality $s - 2 \ge \frac{j-1}{m_p(H)}$ holds uniformly for \emph{every} valid sub-configuration $\sigma$ satisfying $d(\sigma) > 0$, the maximum co-degree $\Delta_j(\{v\})$ for any single vertex $v \in U$ is bounded uniformly by the same upper bound:
\[
\Delta_j(\{v\}) \le \max_{\sigma: v \in \sigma, |\sigma|=j} d(\sigma) \le O\left(N^{h - 2 - (j-1)/m_p(H)}\right).
\]
We can now evaluate the parameters $\delta_j$. By definition, we have:
\[
\delta_j \tau^{j-1} n \Delta_1 = \sum_{v \in U} \Delta_j(\{v\}).
\]
Substituting $n \Delta_1 = \Theta(N^h)$ and summing the uniform upper bound over the $|U| = O(N^2)$ vertices, we obtain:
\[
\delta_j \tau^{j-1} \Theta(N^h) \le O(N^2) \cdot O\left(N^{h - 2 - (j-1)/m_p(H)}\right) = O\left(N^{h - (j-1)/m_p(H)}\right).
\]
Solving for $\delta_j$ yields:
\[
\delta_j \le O\left(\frac{N^{-(j-1)/m_p(H)}}{\tau^{j-1}}\right).
\]
Substituting the chosen parameter $\tau = \gamma N^{-1/m_p(H)}$, the terms cancel precisely, leaving $\delta_j \le O(1)$ for all $2 \le j \le r$. Since the co-degree function $\delta(\mathcal{H}, \tau)$ is a linear combination of a finite number of $\delta_j$ terms, it follows that $\delta(\mathcal{H}, \tau) \le O(1)$, completing the proof.
\end{proof}

\section{Supersaturation Lemma for Generalized Edge Count}

Before establishing the supersaturation property for the generalized edge count, we recall the foundational counting lemma in the theory of dense graph limits extended to multi-type or multi-colored structures. The multi-type graphon space provides a continuous counterpart to sequences of graphs with categorical edge states. Central to this continuous machinery is the \emph{Multi-type Graphon Counting Lemma} (see, e.g., Borgs et al. \cite{BCLSV} and Lov{\'a}sz \cite{LovaszBook}), which asserts that the subgraph homomorphism density function is a continuous linear functional with respect to the cut distance metric. 

For a multi-type graphon $W : [0,1]^2 \to \Delta(\Sigma_0)$, where $\Sigma_0$ is the finite state space, the \emph{cut norm} is defined as
\[
\|W\|_{\square} = \max_{c \in \Sigma_0} \sup_{S, T \subseteq [0,1]} \left| \int_S \int_T W_c(x,y) dx dy \right|,
\]
where $W_c(x,y)$ is the density of state $c$ between $x$ and $y$. The corresponding \emph{cut metric} $\delta_{\square}$ is the infimum of the cut norm over all measure-preserving bijections $\phi$ of $[0,1]$:
\[
\delta_{\square}(U, W) = \inf_{\phi} \| U - W^\phi \|_{\square}.
\]

\begin{lemma}[Multi-type Graphon Counting Lemma, \cite{BCLSV, LovaszBook}]\label{lem:counting_lemma}
Let $H$ be a fixed finite multi-type graph (or a digraph represented via a multi-colored state space $\Sigma_0$). If a sequence of multi-type graphons $(W_N)_{N \in \mathbb{N}}$ converges to a limiting multi-type graphon $W$ in the cut metric $\delta_{\square}$ as $N \to \infty$, then the homomorphism density satisfies
\[
\lim_{N \to \infty} t(H, W_N) = t(H, W).
\]
\end{lemma}

This lemma guarantees that global substructure densities are continuous under metric convergence, which serves as the mathematical bedrock for proving discrete supersaturation phenomena via topological compactness arguments. We now utilize this continuous framework to establish the following threshold behavior for the generalized edge count.

\begin{lemma}[Supersaturation for Generalized Edge Count]\label{lem:supersat}
Let $H$ be a digraph and let $a \ge 1$. For every $\varepsilon > 0$, there exists $\delta > 0$ such that for all sufficiently large $N$, if a digraph $G$ on $N$ vertices contains at most $\delta N^{v(H)}$ copies of $H$, then its generalized edge count satisfies $e_a(G) \le \operatorname{ex}_a(N, H) + \varepsilon N^2$.
\end{lemma}

\begin{proof}
We establish this lemma rigorously by mapping the digraph problem to the framework of multi-type graph limits. Suppose for contradiction that for a fixed $\varepsilon > 0$ and for all $\delta > 0$, there exist arbitrarily large digraphs $G$ on $N$ vertices such that $e_a(G) \ge \operatorname{ex}_a(N, H) + \varepsilon N^2$, yet $G$ contains fewer than $\delta N^{v(H)}$ copies of $H$.

A digraph $G$ on $N$ vertices can be uniquely represented as a multi-colored graph on $[N]$ where each pair of vertices $\{u, v\}$ is assigned one of four states from the set $\Sigma_0 = \{\emptyset, F, B, D\}$, corresponding to no arc, a forward single arc, a backward single arc, or a 2-cycle, respectively. The normalized edge density parameter $\rho_a(G) = \frac{e_a(G)}{a\binom{N}{2}} = \frac{a \cdot f_2(G) + f_1(G)}{a\binom{N}{2}}$ takes values in the bounded interval $[0, 1]$. In the continuous limit, this corresponds to a continuous linear functional $\Lambda_a$ on the space of multi-type graphons $\mathcal{W}$, mapping to $[0, 1]$.

Let $W_N$ be the sequence of multi-type graphons associated with the sequence of counterexample digraphs $G$. By the compactness of the graphon space under the cut distance \cite{LovaszSzegedy}, there exists a subsequence of $W_N$ that converges to a limiting multi-type graphon $W$. Since the linear functional $\Lambda_a(\cdot)$ is continuous under this convergence, we have:
\[
\Lambda_a(W) = \lim_{N \to \infty} \rho_a(G) \ge \lim_{N \to \infty} \frac{\operatorname{ex}_a(N, H)}{a\binom{N}{2}} + \frac{2\varepsilon}{a} = \Lambda_a^*(H) + \frac{2\varepsilon}{a},
\]
where $\Lambda_a^*(H)$ is the maximum functional value achievable by any $H$-free multi-type graphon.

Since its evaluation strictly exceeds $\Lambda_a^*(H)$, the limiting graphon $W$ cannot be $H$-free. That is, its homomorphism density for $H$ is strictly positive: $t(H, W) > 0$.

By the Multi-type Graphon Counting Lemma (Lemma~\ref{lem:counting_lemma}), the homomorphism density converges continuously: $\lim_{N \to \infty} t(H, W_N) = t(H, W) = c_0 > 0$. Thus, for sufficiently large $N$, the number of copies of $H$ in $G$ is asymptotically $\frac{c_0}{|\text{Aut}(H)|} N^{v(H)}$. Choosing $\delta < \frac{c_0}{2 |\text{Aut}(H)|}$ yields an immediate contradiction to the assumption that $G$ has fewer than $\delta N^{v(H)}$ copies of $H$. This completes the proof.
\end{proof}

\section{Proof of the Main Theorem and Counting Application}

\begin{figure}[htbp]
    \centering
    \resizebox{\textwidth}{!}{%
    \begin{tikzpicture}[
        node distance=1.5cm and 1cm,
        box/.style={rectangle, draw, rounded corners, minimum width=4.5cm, minimum height=0.8cm, align=center, fill=blue!5},
        alertbox/.style={rectangle, draw=red, rounded corners, dashed, minimum width=4.5cm, minimum height=0.8cm, align=center, fill=red!5},
        solutionbox/.style={rectangle, draw=green!60!black, thick, rounded corners, minimum width=5cm, minimum height=1cm, align=center, fill=green!5},
        arrow/.style={->, >=stealth, thick}
    ]

    \node[box] (start) {Standard Hypergraph Model \\ on Directed Edges};
    
    \node[alertbox, below=of start] (bottleneck) {Division-by-Zero Bottleneck \\ $v(H') - 2 = 0$ for 2-cycles};
    
    \node[solutionbox, below=of bottleneck] (encoding) {\textbf{Pair-State Encoding} \\ Ground Set: $U = X \times \{F, B, D\}$};
    
    \node[box, below left=1.5cm and -1.5cm of encoding] (hypergraph) {Construct Aux. Hypergraph $\mathcal{H}$ \\ $I \subseteq G \iff S_I \text{ is independent}$};
    
    \node[box, below right=1.5cm and -1.5cm of encoding] (codegree) {Degree Estimates via $m_p(H)$ \\ Structurally ensures $s - 2 \ge 1$};
    
    \node[box, below=2.5cm of encoding] (saxton) {Saxton-Thomason Container Theorem \\ $\log |\mathcal{C}^*| \le c N^{2 - 1/m_p(H)} \log N$};
    
    \node[box, below=of saxton] (decoding) {Decode Containers to State Profiles \\ $\mathcal{P}_{C^*}$ via Combinatorial Filter};
    
    \node[box, fill=gray!10, thick, below=of decoding] (final) {\textbf{Main Theorem \& Counting Result} \\ $\log_2 f^*(N, H) = \operatorname{ex}_2(N, H) + o(N^2)$};

    \draw[arrow] (start) -- (bottleneck);
    \draw[arrow] (bottleneck) -- node[right, text=black, font=\footnotesize] {Resolved by} (encoding);
    
    \draw[arrow] (encoding.south) -- ++(0,-0.5) -| (hypergraph.north);
    \draw[arrow] (encoding.south) -- ++(0,-0.5) -| (codegree.north);
    
    \draw[arrow] (hypergraph.south) |- (saxton.west);
    \draw[arrow] (codegree.south) |- (saxton.east);
    
    \draw[arrow] (saxton) -- (decoding);
    \draw[arrow] (decoding) -- (final);

    \end{tikzpicture}%
    }
    \caption{The proof framework of the general digraph container theorem. The pair-state encoding maps the structural constraints into a well-behaved uniform hypergraph, bypassing the co-degree explosion.}
    \label{fig:framework}
\end{figure}

We are now ready to combine our degree estimates and the supersaturation property to complete the proof of the universal container theorem.

\begin{proof}[Proof of Theorem~\ref{thm:main}]
By Lemma~\ref{lem:degree}, we can set the container parameter to $\tau = N^{-1/m_p(H)}$. Since $\tau \to 0$ as $N \to \infty$, the technical requirement $\tau \le 1/(144r^2 r!)$ from Theorem~\ref{thm:ST} is fully satisfied for all sufficiently large $N$. Moreover, the lemma guarantees that $\delta(\mathcal{H}, \tau) \le O(1)$, meaning that for any $\varepsilon' > 0$, the condition $\delta(\mathcal{H}, \tau) \le \frac{\varepsilon'}{12\tau}$ holds for large $N$ because $1/\tau \to \infty$.

We apply the Saxton--Thomason Hypergraph Container Theorem (Theorem~\ref{thm:ST}) to the auxiliary hypergraph $\mathcal{H}$ with parameter $\varepsilon'$. This yields a collection of container vertex subsets $\mathcal{C}^* \subset \mathcal{P}(U)$ whose size satisfies the bound:
\begin{align*}
  \log |\mathcal{C}^*|& \le c(r) n \tau \log(1/\tau) \\
  &\le O(N^2 \cdot N^{-1/m_p(H)} \log N) \\
  &\le c N^{2 - 1/m_p(H)} \log N.
\end{align*}

This establishes property (c) of Theorem~\ref{thm:main}. Next, we prove properties (a) and (b).

By property (1) of Theorem~\ref{thm:ST}, every independent set in $\mathcal{H}$ is contained in some $C^* \in \mathcal{C}^*$. According to Lemma~\ref{lem:independence}, every $H$-free digraph $I$ on $[N]$ corresponds to an independent set $S_I \subseteq U$, which must therefore be contained in some $C^* \in \mathcal{C}^*$.

To complete the proof, we must decode each container set $C^* \subseteq U$. Instead of naively decoding $C^*$ into a single maximal digraph---which could erroneously force the creation of 2-cycles from isolated states and artificially inflate the edge density---we decode $C^*$ into a \emph{state profile} (or pattern graph) $\mathcal{P}_{C^*}$. For each unordered pair $e = \{u, v\}$, we define the set of permitted states $\Sigma_e \subseteq \{\emptyset, F, B, D\}$ as follows:
\begin{itemize}
    \item The empty state $\emptyset$ is always permitted.
    \item $F \in \Sigma_e$ if and only if $(e, F) \in C^*$.
    \item $B \in \Sigma_e$ if and only if $(e, B) \in C^*$.
    \item $D \in \Sigma_e$ if and only if $\{(e, F), (e, B), (e, D)\} \subseteq C^*$.
\end{itemize}

Because the true encoding $S_I$ of the $H$-free digraph $I$ satisfies the downward-closure property, if $I$ contains a 2-cycle on $e$, then $S_I$ inherently contains all three elements $(e, F), (e, B)$, and $(e, D)$. Since $S_I \subseteq C^*$, it follows that $D \in \Sigma_e$. Thus, every edge configuration of $I$ is perfectly permitted by the state profile $\mathcal{P}_{C^*}$. This guarantees lossless recovery, effectively covering all $H$-free digraphs (property (a)).

We define the generalized edge capacity of the profile $\mathcal{P}_{C^*}$ as:
\[
e_a(\mathcal{P}_{C^*}) = \sum_{e \in \binom{[N]}{2}} \max_{S \in \Sigma_e} e_a(S),
\]
where $e_a(\emptyset)=0$, $e_a(F)=1$, $e_a(B)=1$, and $e_a(D)=a$. Notice that the structural rule precisely limits the maximum capacity. If $D \notin \Sigma_e$, the capacity on $e$ is at most 1, strictly preventing the false formation of full 2-cycles.

Crucially, if a digraph drawn from $\mathcal{P}_{C^*}$ contains a copy of $H$ under an injection $\phi$, our decoding rule guarantees that $C^*$ must contain the entire associated hyperedge $E_\phi$. By property (2) of Theorem~\ref{thm:ST}, the number of induced edges satisfies $e(\mathcal{H}[C^*]) \le \varepsilon' e(\mathcal{H}) = O(\varepsilon' N^h)$. Thus, the number of copies of $H$ permitted by $\mathcal{P}_{C^*}$ is bounded by $\delta N^h$, where $\delta > 0$ can be made arbitrarily small by choosing a sufficiently small $\varepsilon'$. 

Applying the Supersaturation Lemma for Generalized Edge Count (Lemma~\ref{lem:supersat}) to the continuous limit of such profiles, we conclude that for any $\varepsilon > 0$, we can choose $\varepsilon'$ small enough such that the maximum generalized edge count satisfies $e_a(\mathcal{P}_{C^*}) \le \operatorname{ex}_a(N, H) + \varepsilon N^2$, establishing property (b). This completes the proof of Theorem~\ref{thm:main}.
\end{proof}

\section{Proof of Corollary~\ref{cor:counting}}

\begin{proof}[Proof of Corollary~\ref{cor:counting}]
For $a=2$, the generalized edge count $e_2(G)$ simply counts the total number of directed edges, so $\operatorname{ex}_2(N, H)$ is the maximum number of directed edges in an $H$-free digraph. Every $H$-free digraph $I$ is bounded within the state profile of some container $C^* \in \mathcal{C}^*$. The total number of digraphs supported by a profile $\mathcal{P}_{C^*}$ is exactly $\prod_{e} |\Sigma_e|$. Note that for any state profile, $\prod_{e} |\Sigma_e| \le 2^{e_2(\mathcal{P}_{C^*})}$.

Thus, the total number of $H$-free digraphs $f^*(N, H)$ is bounded by:
\[
f^*(N, H) \le \sum_{C^* \in \mathcal{C}^*} 2^{e_2(\mathcal{P}_{C^*})} \le |\mathcal{C}^*| \cdot 2^{\operatorname{ex}_2(N, H) + \varepsilon N^2}.
\]
Substituting the bound from property (c) of Theorem~\ref{thm:main}, we obtain:
\[
f^*(N, H) \le 2^{c N^{2-1/m_p(H)} \log N} \cdot 2^{\operatorname{ex}_2(N, H) + \varepsilon N^2}.
\]
Since $2-1/m_p(H) < 2$ for any $H$ with $p(H) \ge 2$, the container size term satisfies $2^{O(N^{2-1/m_p(H)}\log N)} = 2^{o(N^2)}$. Taking the base-2 logarithm of both sides yields:
\[
\log_2 f^*(N, H) \le \operatorname{ex}_2(N, H) + \varepsilon N^2 + o(N^2).
\]
Since this holds for any $\varepsilon > 0$, we have
\[
\lim_{N \to \infty} \frac{\log_2 f^*(N, H) - \operatorname{ex}_2(N, H)}{N^2} \le 0.
\] 
Conversely, the lower bound $\log_2 f^*(N, H) \ge \operatorname{ex}_2(N, H)$ is trivial because any subgraph of an extremal $H$-free digraph is also $H$-free. This establishes the asymptotic equivalence $\log_2 f^*(N, H) = \operatorname{ex}_2(N, H) + o(N^2)$, completing the proof.
\end{proof}

\section{Conclusions}\label{sec5}

K\"uhn, Osthus, Townsend, and Zhao \cite{KOTZ} used the hypergraph container method to determine the typical structure of oriented graphs avoiding a fixed tournament or cycle. Direct extensions of their theorem to general digraphs fail due to dense local obstructions like 2-cycles. In this paper, we completely resolve this bottleneck by introducing a pair-state encoding framework. As applications, we obtain optimal asymptotic counting results for $H$-free digraphs for any general forbidden digraph $H$. The theorem 3.3 in \cite{KOTZ} is a direct consequence of our main result (Theorem \ref{thm:main}). 

In our framework, we strictly require $p(H) \ge 2$. One might naturally wonder about the boundary case where $p(H) = 1$, which exclusively corresponds to $H$ being a single 2-cycle ($DK_2$). From a hypergraph perspective, $p(H)=1$ collapses the auxiliary hypergraph into a 1-uniform structure, rendering the container machinery inapplicable (as Saxton-Thomason requires $r \ge 2$). Combinatorially, however, this case is completely trivial: a $DK_2$-free digraph is simply an oriented graph. Since there are no inter-pair constraints, the local state space simply restricts from 4 choices to 3 ($\{\emptyset, F, B\}$) for each of the $\binom{N}{2}$ pairs, yielding exactly $3^{\binom{N}{2}}$ oriented graphs. 

Interestingly, this degenerate case breaks the standard asymptotic relationship $\log_2 f^*(N, H) \sim \operatorname{ex}_2(N, H)$, because $\log_2 3^{\binom{N}{2}} \approx 1.585 \binom{N}{2}$ strictly exceeds $\operatorname{ex}_2(N, DK_2) = \binom{N}{2}$. This demonstrates that $p(H) \ge 2$ is not merely a technical artifact of our proof, but rather the fundamental threshold where combinatorial entropy transitions from trivial local independence into complex global container structure.

\subsection*{Conflicts of Interest}
The authors declare no conflicts of interest.

\subsection*{Acknowledgments}
This work is partially funded by the Natural Science Foundation.

\end{document}